\documentclass{l4dc2024}
\usepackage{textcomp}
\usepackage[english]{babel}
\usepackage[autostyle=true]{csquotes}
\usepackage{grffile}
\graphicspath{{Images/}}
\usepackage{floatrow}
\usepackage{subcaption}
\DeclareCaptionFormat{custom}
{
    \textbf{\small #1#2}{\small #3}
}
\usepackage[%
    olditem,  % Do not modify itemize environments by default
    oldenum   % Do not modify enumerate environments by default
  ]{paralist}
\usepackage{booktabs}
\usepackage{multirow}
\usepackage{mathrsfs}
\usepackage{bm}
\usepackage{mathtools}
\usepackage{newtxmath}
\usepackage[cal=cm]{mathalpha}
\allowdisplaybreaks
\usepackage{times}
\usepackage{hyperref}
\usepackage{bookmark}
\hypersetup{hidelinks}
\usepackage[noabbrev,nameinlink,sort]{cleveref}
\usepackage{crossreftools}
\usepackage{import}

\newcommand{\R}{\mathbb{R}}             % Real numbers
\newcommand{\Rp}{\R_{\geq 0}}           % Non-negative reals
\newcommand{\Rpp}{\R_{>0}}              % Positive reals
\newcommand{\Np}{\N_{>0}}               % Positive naturals
\newcommand{\kinf}{\mathcal{K}_\infty}  % Class kappa-infinity functions
\newcommand{\N}{\mathbb{N}}             % Natural numbers

\newcommand{\rs}[1]{#1}
\newcommand{\rz}[1]{#1}

\newcommand{\fw}{{f}}                   % True dynamics
\newcommand{\f}{{f}}                    % Nominal dynamics
\newcommand{\w}{{\mathbb{W}}}           % Uncertainty
\newcommand{\X}{{\mathscr{X}}}          % Set of closed loop states
\newcommand{\XX}{{\mathcal{X}}}         % Tubes
\newcommand{\ZZ}{{\mathscr{Z}}}         % Set of closed loop state-input pairs
\newcommand{\U}{{\mathscr{U}}}          % Feasible input space
\renewcommand{\l}{\ell}                 % Stage cost
\newcommand{\D}{{\mathcal{D}}}          % Data set
\newcommand{\1}{{\bm{1}}}               % Vector full of ones

\newlength{\len}
\newcommand{\alphaone}{\alpha_{1}}
\newcommand{\alphaeleven}{\alpha_{3}}
\newcommand{\alphatwelve}{\alpha_{4}}
\newcommand{\alphathree}{\alpha_{5}}
\newcommand{\alphafour}{\alpha_{6}}
\newcommand{\alphatwo}{\alpha_{2}}
\newcommand{\alphafive}{\alpha_{7}}
\newcommand{\alphasix}{\alpha_{8}}
\newcommand{\alphaseven}{\alpha_{9}}
\newcommand{\alphaeight}{\alpha_{10}}
\newcommand{\alphanine}{\alpha_{11}}
\newcommand{\alphaten}{\alpha_{12}}

\newcommand{\deltaone}{\delta_{1}}
\newcommand{\deltatwo}{\delta_{2}}               % Math snippets
\usepackage{pgfplots}
\pgfplotsset{compat=newest}
\usetikzlibrary{plotmarks}
\usetikzlibrary{shapes,arrows,decorations.markings}
\usetikzlibrary{arrows.meta}
\usepgfplotslibrary{patchplots}
\usetikzlibrary{positioning}
\usetikzlibrary{calc}
\usepackage{tkz-euclide}
\pgfplotsset{tick scale binop=\times}
\newtheorem{assumption}{Assumption}
\crefname{assumption}{Assumption}{Assumptions}
\crefname{lemma}{Lemma}{Lemmas}
\crefname{section}{Section}{Sections}
\crefname{subsection}{Section}{Sections}
\crefname{proposition}{Proposition}{Propositions}
\crefname{equation}{}{}
\crefmultiformat{equation}{(#2#1#3)}{,\,(#2#1#3)}{,\,(#2#1#3)}{,\,(#2#1#3)}\crefrangeformat{equation}{(#3#1#4~to~#5#2#6)}
\crefname{algorithm}{algorithm}{algorithms}
\crefname{prob}{problem}{problems}
\crefname{theorem}{Theorem}{Theorems}
\crefname{remark}{Remark}{Remark}
\crefname{appendix}{Appendix}{Appendices}
\crefname{figure}{Figure}{Figures}        % Custom environments
\begin{document}

% Title
\title[Convergence guarantees for for adaptive MPC with kinky inference]{Convergence guarantees for adaptive model predictive control with kinky inference}

% Authors
\author{%
 \Name{Riccardo Zuliani} \Email{rzuliani@ethz.ch}
 \AND
 \Name{Raffaele Soloperto} \Email{soloperr@ethz.ch}
 \AND
 \Name{John Lygeros} \Email{jlygeros@ethz.ch}\\
 \addr All the authors are with the Automatic Control Laboratory, ETH Z\"urich%
}

% Print title
\maketitle
% Abstract
\begin{abstract}
We analyze the convergence properties of a robust adaptive model predictive control algorithm used to control an unknown nonlinear system. We show that by employing a standard quadratic stabilizing cost function, and by recursively updating the nominal model through \textit{kinky inference}, the resulting controller ensures convergence of the true system to the origin, despite the presence of model uncertainty. We illustrate our theoretical findings through a numerical simulation.
\end{abstract}
% Keywords
% Acknowledgements
%%%%%%%%%%%%%%%%%%%%%%%%%%%%%%%%%%%%%%%%%%%%%%%%%%%%%%%%%%%%%%%%%%%%%%%%
% Text sources
%%%%%%%%%%%%%%%%%%%%%%%%%%%%%%%%%%%%%%%%%%%%%%%%%%%%%%%%%%%%%%%%%%%%%%%%
\section{Introduction}

\paragraph*{Motivation}

During the last few decades, model predictive control (MPC) has attracted large attention because of its efficiency in handling nonlinear systems subject to hard state and input constraints, while minimizing a user-defined cost function, \cite{rawlings2017model}. An MPC scheme employs a nominal model of the system dynamics to predict future trajectories over a given prediction horizon. However, in several applications, obtaining an accurate model can be expensive in terms of money and resources, \cite{darby2012mpc}. Model inaccuracies, combined with the presence of disturbances affecting the system, might lead to a deterioration of performance, constraint violation, or even instability, \cite{forbes2015model}.

To address this issue, research has focused on adaptive MPC schemes where the identification of the model is performed online together with the computation of the input. Since excitation is not explicitly enforced, such approaches are often referred to as passive-learning controllers, \cite{mesbah2018stochastic}.
One of the main drawbacks of passive-learning approaches is that newly generated data might not be informative (for example, when the system reaches a steady state), and therefore an improvement in the model estimate is not guaranteed. This issue is addressed in active-learning approaches where a form of excitation is explicitly induced, often by incorporating a learning cost into the MPC cost function, \cite{tanaskovic2019adaptive}, \cite{soloperto2019dual}. Even though excitation is beneficial for model adaption, it can be counterproductive in terms of stability. Motivated by this, we analyze the stability properties of an adaptive MPC scheme.

%These schemes are sometimes called ``dual'' MPC schemes, since the control action is used to fulfill two contrasting objectives: stabilizing the system while minimizing the control objective, and providing sufficient excitation to improve the accuracy of the model \cite{heirung2017dual}.

\paragraph*{Related works}
% \color{red}
%In the context of Adaptive MPC for uncertain linear parameter varying systems, \cite{soloperto2019dual} proposes an MPC scheme that utilizes set membership identification combined with recursive least squares regression to obtain both a nominal parameter estimate and a bound on the uncertainty, so that robust constraint satisfaction can be ensured.  proposes a similar method for time-varying systems. In \cite{parsi2022explicit} the authors improve an existing robust adaptive MPC solution by performing online updates of the terminal set and terminal control law, as well as allowing the tracking setpoint to change over time. These robust adaptive approaches are not limited to linear systems. The set-based method proposed in \cite{canale2014nonlinear} is applicable to unknown nonlinear Lipschitz continuous systems. In \cite{kohler2021robust} the authors introduce a computationally tractable tube-based adaptive MPC scheme with guaranteed constraint satisfaction which operates on nonlinear systems subject to both parametric uncertainty and additive disturbances.

% \color{black}

%Generally, however, dual MPC schemes achieve the objective of learning the system model thanks to the introduction of an additional term in the objective function, see \cite{soloperto2022guaranteed,soloperto2020augmenting}, \cite{heirung2017dual}. This term, sometimes called learning cost, may lead to deteriorated performance, and it represents a source of additional computational complexity.

A Gaussian process (GP) is a collection of random variables, any finite subset of which follows a multivariate Gaussian distribution, \cite{rasmussen2003gaussian}. In the context of adaptive MPC, GPs are employed to construct a nominal model and sets that bound, within a certain probability, the evolution of the true system and the learned one, \cite{hewing2019cautious}, \cite{bradford2020stochastic}. However, in the case where robust constraint satisfaction is needed, e.g., in safety-critical systems, GP-based adaptive MPC schemes cannot be employed as they fail to ensure hard constraint satisfaction.

Conversely, robust adaptive MPC approaches use a set-membership method to construct robust sets where the uncertainty lies, and then incorporate this knowledge into a robust tube-based MPC scheme, see, e.g., \cite{lorenzen2019robust}, \cite{lu2019robust}, \cite{tanaskovic2019adaptive}, \cite{kohler2021robust}, \cite{soloperto2019dual} for the case of systems subject to parametric uncertainty. If the system is subject to non-parametric uncertainty, then \textit{kinky inference} is a learning method that is able to learn an unknown nonlinear system, while providing uncertainty sets that robustly bound the mismatch between the obtained estimate and the actual system, \cite{calliess2014conservative}.
In \cite{calliess2020lazily}, the authors studied the theoretical properties of kinky inference in the context of regression for simple adaptive control problems.
Kinky inference has been successfully combined with model predictive control; for example, in \cite{limon2017learning}, a learning-based MPC scheme uses a kinky inference to model an unknown system while providing safety guarantees, and input-to-state stability. In \cite{manzano2019output}, the authors propose a smoothed version of the method that is more suitable for real-time control thanks to its lower computation effort. In this case, only soft constraints on the outputs are imposed. The approach is subsequently improved in \cite{manzano2022input}, where a kinky inference-based predictor that is guaranteed to be Lipschitz continuous is considered. Even though the authors show that the origin is input-to-state stable, convergence is not shown.

\paragraph*{Contribution}
In this paper, we perform a stability analysis of a robust adaptive MPC scheme where the model is learned through a kinky inference.
In particular, in addition to input-to-state stability, we show that standard robust adaptive MPC schemes can successfully ensure convergence of the closed-loop to the origin, despite the presence of model uncertainty. This theoretical result is obtained by only employing a tracking cost function, without the need to enforce any kind of excitation in the system, i.e., in a passive-learning fashion.

\paragraph*{Outline}

\cref{section:PF} introduces the unknown system under consideration, shows how kinky inference can be used for modeling, and discusses the MPC scheme used for control. \cref{section:TA} demonstrates that the closed-loop dynamics are input-to-state stable, and that, in addition, the closed-loop system asymptotically converges to the origin. Simulations are presented in \cref{section:EX}, while Section \ref{section:conclusion} concludes the paper.

\paragraph*{Notation}
We use $\Rp$ and $\Rpp$ to represent the sets of non-negative and positive real numbers, respectively. A function $\alpha: \R_{\geq0} \rightarrow \R_{\geq 0}$ is a class $\kinf$ function, i.e., $\alpha \in \kinf$ if $\alpha$ is strictly increasing, $\alpha(0)=0$, and $\lim_{t\rightarrow \infty} \alpha(t) = \infty$. \rz{$\|\cdot\|$ denotes the Euclidean norm.} $\1$ is a vector of appropriate dimension where each entry is equal to $1$.
\newpage

\section{Problem Setup}%
\label{section:PF}
%%%%%%%%%%%%%%%%%%%%%%%%%%%%%%%%%%%%%%%%%%%%%%%%%%%%%%%%%%%%%%%%%%%%%%%%%%%%%%%%%%%%%%%%%%%%%%%%%%%%%%%%%%%%%%
\paragraph{System description:}\label{subsection:PF_system_description} Consider the nonlinear, time-invariant system described, for each time $t\in\N$, as follows
\begin{align}
x_{t+1}=\fw(x_t,u_t),%
\label{eq:PF_true_dynamics}
\end{align}
where $x_t\in\R^n$ and $u_t\in\R^m$ denote the state and the input of the system at time $t$, respectively. The state $x_t$ is assumed to be fully available for measurement at each $t\in\N$. The system is subject to the following state and input constraints
\begin{align}
\label{eq:constraints}
(x_t,u_t)\in \X \times \U =: \ZZ,   
\end{align}
for each $t\in\N$, where $\X$ and $\U$ are known convex and compact sets that contain the origin in their interior. To simplify the notation, we set $z:=(x,u)$ \rs{when referring to a generic state-input pair}, and $z_t:=(x_t,u_t)$ \rs{for the state-input pair at time $t$}. The function $\fw:\R^n \times \R^m \to \R^n$ is unknown, but satisfies the following Assumption.
\begin{assumption}%
\label{ass:PF_system_assumption}
The function $\fw:\rz{\ZZ} \to \R^n$ is H\"older continuous in all of its arguments, i.e., there exist known constants $q\in\Rpp$ and $\lambda\in\R$, with $0<\lambda\leq 1$, such that
\begin{align*}
\|\fw(z_1)-\fw(z_2)\| \leq q \|z_1-z_2\|^\lambda,
\end{align*}
holds for all $z_1,z_2\in\ZZ$. Moreover, the origin is an equilibrium point for $\fw$, i.e., $\fw(0,0)=0$.
\end{assumption}
%
%\begin{remark}
Note that the case where $\lambda>1$ implies that the underlying function is constant on $\ZZ$ (see Proposition 1.1.16, \cite{fiorenza2017holder}). \rs{Estimating the H\"older constants $q$ and $\lambda$ from data can be done e.g., from first principles following the techniques outlined in Section 2.5 of \cite{calliess2014conservative} or in  \cite{huang2023sample} for the Lipschitz case}.
%\end{remark}
%
%%%%%%%%%%%%%%%%%%%%%%%%%%%%%%%%%%%%%%%%%%%%%%%%%%%%%%%%%%%%%%%%%%%%%%%%%%%%%%%%%%%%%%%%%%%%%%%%%%%%%%%%%%%%%%
\paragraph{Kinky inference: }%
\label{subsection:PF_kinky_inference}%
\emph{Kinky inference} is a learning technique that can be applied to H\"older continuous functions, \cite{calliess2014conservative}. We use kinky inference to construct a nominal model $\f_t: \rz{\ZZ} \to \R^n$ and an uncertainty function $\w_t: \rz{\ZZ} \to 2^{\R^n}$. 
For each $t\in\Np$, we define the \emph{data set} $\mathcal{D}_t \subset \R^{n+m}$ iteratively by $\mathcal{D}_t := \mathcal{D}_{t-1} \cup \left\{ z_{t-1} \right\}$, with $\mathcal{D}_0:= \left\{ 0 \right\}$. To construct the uncertainty functions $\w_t$, we define the confidence bounds $w_t^\text{max},w_t^\text{min}:\R^{n} \times \R^m \to \R^n$:
\begin{align}
\label{eq:PF_kinky_bounds_max} 
w_t^\text{max}(z) := \min_{y\in\D_t} \fw(y) + \1 q \|z-y\|^\lambda, \qquad w_t^\text{min}(z) := \max_{y\in\D_t} \fw(y) - \1 q \|z-y\|^\lambda. \end{align}
Since the system $f$ satisfies \cref{ass:PF_system_assumption}, it is possible to show that
\begin{align}
\label{eq:PF_true_system_bounds}
w_t^\text{min}(z) \leq \fw(z) \leq w_t^\text{max}(z), \quad \forall z\in \ZZ,
\end{align}
where the inequalities are meant element-wise. 
% \color{red}
\begin{assumption}
\label{ass:nominal_model}
At each time $t\in \N$, the nominal model $f_t$ is chosen as a H\"older continuous function satisfying the following conditions
\begin{subequations}
\begin{align}
w_t^{\min}(z) \leq \f_t(z) \leq  w_t^{\max}(z),~~~ &\forall z\in\ZZ,\label{eq:PF_nominal_model}\\
\| f_t(z)-f_t(y) \| \leq \alphaone(\| z-y \|),~~~ &\forall z,y\in \ZZ, ~~~ \forall t\in\N, \label{eq:PF_nominal_model_1}\\
\| f_{t}(z)-f_{t+1}(z) \| \leq \alphatwo( \| w^{\max}_{t}(z_{t})-w^{\min}_{t}(z_{t}) \|), ~~~ &\forall z\in\ZZ, ~\forall t\in \N.\label{eq:TA_model_updates_bounded_1} 
\end{align}
\end{subequations}
with $\alphaone, \alphatwo \in \mathcal{K}_{\infty}$.
\end{assumption}
Assumption \ref{ass:nominal_model} is satisfied if the nominal model $\f_t$ is chosen as the average between the two confidence bounds, as depicted in \cref{fig:PF_gp_figure} for the case of a one-dimensional system. Note that condition \cref{eq:TA_model_updates_bounded_1} implies that for any $z\in \ZZ$, the model update is bounded above by a function $\alphatwo$ that solely depends on the most recently visited state-input pair $z_t$, and not the considered state $z$. 
\begin{lemma}
    Condition \eqref{eq:TA_model_updates_bounded_1} can be equivalently re-stated as follows
\begin{align*}
    \| f_{t}(z)-f_{t+1}(z) \| \leq \alphatwo(\max_{y\in \ZZ} \| w_{t}^{\max}(y)-w_{t}^{\min}(y) - (w_{t+1}^{\max}(y)-w_{t+1}^{\min}(y)) \|), ~~~ &\forall z\in\ZZ, ~\forall t\in \N.
\end{align*}
\end{lemma}
\begin{figure}[t!]
\floatbox[{\capbeside\thisfloatsetup{capbesideposition={left,center},capbesidewidth=5cm}}]{figure}[\FBwidth]
{\caption{Kinky inference bounds for $\lambda=1$ (thin continuous lines) and mean function (dashed line, in red) compared against the true function (thick and continuous line, in blue) for a one-dimensional system.}\label{fig:PF_gp_figure}}
{% This file was created by matlab2tikz.
%
%The latest updates can be retrieved from
%  http://www.mathworks.com/matlabcentral/fileexchange/22022-matlab2tikz-matlab2tikz
%where you can also make suggestions and rate matlab2tikz.
%
\definecolor{mycolor1}{rgb}{0.00000,0.44700,0.74100}%
\definecolor{mycolor2}{rgb}{0.85000,0.32500,0.09800}%
\definecolor{mycolor3}{rgb}{0.92900,0.69400,0.12500}%
\definecolor{mycolor4}{rgb}{0.63500,0.07800,0.18400}%
\begin{tikzpicture}

\begin{axis}[%
width=0.55*5.877in,
height=0.25*4.635in,
at={(0.986in,0.626in)},
scale only axis,
xmin=0,
xmax=1,
xtick={  0, 0.2, 0.4, 0.6, 0.8,   1},
xlabel style={font=\color{white!15!black}},
xlabel={$x$},
ymin=-0.2,
ymax=1,
ylabel style={font=\color{white!15!black}},
ylabel={$f(x)$},
axis background/.style={fill=white},
axis x line*=bottom,
axis y line*=left,
xmajorgrids,
ymajorgrids,
legend style={at={(0.03,0.97)}, anchor=north west, legend columns=4, align=left, draw=white!15!black},
label style = {font=\small},
tick label style = {font=\small},
legend style = {font=\small}
]
\addplot [color=mycolor1, line width=1.5pt]
  table[row sep=crcr]{%
0	0\\
0.01	9e-05\\
0.02	0.00036\\
0.03	0.00081\\
0.04	0.00144\\
0.05	0.00225\\
0.06	0.00324\\
0.07	0.00441\\
0.08	0.00576\\
0.09	0.00729\\
0.1	0.009\\
0.11	0.01089\\
0.12	0.01296\\
0.13	0.01521\\
0.14	0.01764\\
0.15	0.02025\\
0.16	0.02304\\
0.17	0.02601\\
0.18	0.02916\\
0.19	0.03249\\
0.2	0.036\\
0.21	0.03969\\
0.22	0.04356\\
0.23	0.04761\\
0.24	0.05184\\
0.25	0.05625\\
0.26	0.06084\\
0.27	0.06561\\
0.28	0.07056\\
0.29	0.07569\\
0.3	0.081\\
0.31	0.08649\\
0.32	0.09216\\
0.33	0.09801\\
0.34	0.10404\\
0.35	0.11025\\
0.36	0.11664\\
0.37	0.12321\\
0.38	0.12996\\
0.39	0.13689\\
0.4	0.144\\
0.41	0.15129\\
0.42	0.15876\\
0.43	0.16641\\
0.44	0.17424\\
0.45	0.18225\\
0.46	0.19044\\
0.47	0.19881\\
0.48	0.20736\\
0.49	0.21609\\
0.5	0.225\\
0.51	0.23409\\
0.52	0.24336\\
0.53	0.25281\\
0.54	0.26244\\
0.55	0.27225\\
0.56	0.28224\\
0.57	0.29241\\
0.58	0.30276\\
0.59	0.31329\\
0.6	0.324\\
0.61	0.33489\\
0.62	0.34596\\
0.63	0.35721\\
0.64	0.36864\\
0.65	0.38025\\
0.66	0.39204\\
0.67	0.40401\\
0.68	0.41616\\
0.69	0.42849\\
0.7	0.441\\
0.71	0.45369\\
0.72	0.46656\\
0.73	0.47961\\
0.74	0.49284\\
0.75	0.50625\\
0.76	0.51984\\
0.77	0.53361\\
0.78	0.54756\\
0.79	0.56169\\
0.8	0.576\\
0.81	0.59049\\
0.82	0.60516\\
0.83	0.62001\\
0.84	0.63504\\
0.85	0.65025\\
0.86	0.66564\\
0.87	0.68121\\
0.88	0.69696\\
0.89	0.71289\\
0.9	0.729\\
0.91	0.74529\\
0.92	0.76176\\
0.93	0.77841\\
0.94	0.79524\\
0.95	0.81225\\
0.96	0.82944\\
0.97	0.84681\\
0.98	0.86436\\
0.99	0.88209\\
1	0.9\\
};
\addlegendentry{$f$}

\addplot [color=mycolor2, line width=0.7pt]
  table[row sep=crcr]{%
0	-4.29407774050453e-12\\
0.01	-0.017378414444365\\
0.02	-0.0349940643867516\\
0.03	-0.0149940643867516\\
0.04	-0.00238589459597314\\
0.05	-0.0223858945959731\\
0.06	-0.0423858945959731\\
0.07	-0.0623858945959732\\
0.08	-0.0823858945959731\\
0.09	-0.102385894595973\\
0.1	-0.122385894595973\\
0.11	-0.142385894595973\\
0.12	-0.133630221741604\\
0.13	-0.113630221741604\\
0.14	-0.0936302217416035\\
0.15	-0.0736302217416036\\
0.16	-0.0536302217416036\\
0.17	-0.0336302217416036\\
0.18	-0.0136302217416036\\
0.19	0.00636977825839643\\
0.2	0.0263697782583964\\
0.21	0.0299343066369928\\
0.22	0.00993430663699282\\
0.23	-0.0100656933630072\\
0.24	-0.0300656933630072\\
0.25	-0.0500656933630072\\
0.26	-0.0700656933630072\\
0.27	-0.0900656933630072\\
0.28	-0.110065693363007\\
0.29	-0.130065693363007\\
0.3	-0.150065693363007\\
0.31	-0.130702667905351\\
0.32	-0.110702667905351\\
0.33	-0.0907026679053511\\
0.34	-0.0707026679053511\\
0.35	-0.0507026679053511\\
0.36	-0.0307026679053512\\
0.37	-0.0107026679053512\\
0.38	0.00929733209464884\\
0.39	0.0292973320946489\\
0.4	0.0492973320946489\\
0.41	0.0692973320946489\\
0.42	0.0892973320946488\\
0.43	0.109297332094649\\
0.44	0.129297332094649\\
0.45	0.149297332094649\\
0.46	0.169297332094649\\
0.47	0.189297332094649\\
0.48	0.202484932094649\\
0.49	0.182484932094649\\
0.5	0.162484932094649\\
0.51	0.142484932094649\\
0.52	0.122484932094649\\
0.53	0.102484932094649\\
0.54	0.1002969\\
0.55	0.1202969\\
0.56	0.1402969\\
0.57	0.1602969\\
0.58	0.1802969\\
0.59	0.2002969\\
0.6	0.2202969\\
0.61	0.2402969\\
0.62	0.2602969\\
0.63	0.2802969\\
0.64	0.3002969\\
0.65	0.3202969\\
0.66	0.3402969\\
0.67	0.3602969\\
0.68	0.3802969\\
0.69	0.4002969\\
0.7	0.4202969\\
0.71	0.4402969\\
0.72	0.4602969\\
0.73	0.4762969\\
0.74	0.4562969\\
0.75	0.4362969\\
0.76	0.449\\
0.77	0.469\\
0.78	0.489\\
0.79	0.509\\
0.8	0.529\\
0.81	0.549\\
0.82	0.569\\
0.83	0.589\\
0.84	0.609\\
0.85	0.629\\
0.86	0.649\\
0.87	0.669\\
0.88	0.689\\
0.89	0.709\\
0.9	0.729\\
0.91	0.72\\
0.92	0.74\\
0.93	0.76\\
0.94	0.78\\
0.95	0.8\\
0.96	0.82\\
0.97	0.84\\
0.98	0.86\\
0.99	0.88\\
1	0.9\\
};
\addlegendentry{$w_t^\text{min}$}

\addplot [color=mycolor3, line width=0.7pt]
  table[row sep=crcr]{%
0	4.29407774051283e-12\\
0.01	0.0173815035210845\\
0.02	0.0373815035210845\\
0.03	0.0176141054040269\\
0.04	0.00500593561324838\\
0.05	0.0250059356132484\\
0.06	0.0450059356132484\\
0.07	0.0650059356132484\\
0.08	0.0850059356132484\\
0.09	0.105005935613248\\
0.1	0.125005935613248\\
0.11	0.145005935613248\\
0.12	0.165005935613248\\
0.13	0.185005935613248\\
0.14	0.169934306636993\\
0.15	0.149934306636993\\
0.16	0.129934306636993\\
0.17	0.109934306636993\\
0.18	0.0899343066369928\\
0.19	0.0699343066369928\\
0.2	0.0499343066369928\\
0.21	0.0463697782583964\\
0.22	0.0663697782583964\\
0.23	0.0863697782583965\\
0.24	0.106369778258396\\
0.25	0.126369778258396\\
0.26	0.146369778258396\\
0.27	0.166369778258396\\
0.28	0.186369778258396\\
0.29	0.206369778258396\\
0.3	0.226369778258396\\
0.31	0.246369778258396\\
0.32	0.266369778258396\\
0.33	0.286369778258396\\
0.34	0.306369778258396\\
0.35	0.326369778258396\\
0.36	0.346369778258396\\
0.37	0.366369778258396\\
0.38	0.386369778258396\\
0.39	0.382484932094649\\
0.4	0.362484932094649\\
0.41	0.342484932094649\\
0.42	0.322484932094649\\
0.43	0.302484932094649\\
0.44	0.282484932094649\\
0.45	0.262484932094649\\
0.46	0.242484932094649\\
0.47	0.222484932094649\\
0.48	0.209297332094649\\
0.49	0.229297332094649\\
0.5	0.249297332094649\\
0.51	0.269297332094649\\
0.52	0.289297332094649\\
0.53	0.309297332094649\\
0.54	0.329297332094649\\
0.55	0.349297332094649\\
0.56	0.369297332094649\\
0.57	0.389297332094649\\
0.58	0.409297332094649\\
0.59	0.429297332094649\\
0.6	0.449297332094649\\
0.61	0.469297332094649\\
0.62	0.489297332094649\\
0.63	0.509297332094649\\
0.64	0.529297332094649\\
0.65	0.549297332094649\\
0.66	0.569297332094649\\
0.67	0.589297332094649\\
0.68	0.5762969\\
0.69	0.5562969\\
0.7	0.5362969\\
0.71	0.5162969\\
0.72	0.4962969\\
0.73	0.4802969\\
0.74	0.5002969\\
0.75	0.5202969\\
0.76	0.5402969\\
0.77	0.5602969\\
0.78	0.5802969\\
0.79	0.6002969\\
0.8	0.6202969\\
0.81	0.6402969\\
0.82	0.6602969\\
0.83	0.6802969\\
0.84	0.7002969\\
0.85	0.7202969\\
0.86	0.7402969\\
0.87	0.7602969\\
0.88	0.769\\
0.89	0.749\\
0.9	0.729\\
0.91	0.749\\
0.92	0.769\\
0.93	0.789\\
0.94	0.809\\
0.95	0.829\\
0.96	0.849\\
0.97	0.869\\
0.98	0.889\\
0.99	0.909\\
1	0.9\\
};
\addlegendentry{$w_t^\text{max}$}

\addplot [color=mycolor4, dashed, line width=1.5pt]
  table[row sep=crcr]{%
0	4.14882775983628e-24\\
0.01	1.54453835974483e-06\\
0.02	0.00119371956716644\\
0.03	0.00131002050863762\\
0.04	0.00131002050863762\\
0.05	0.00131002050863762\\
0.06	0.00131002050863762\\
0.07	0.00131002050863762\\
0.08	0.00131002050863763\\
0.09	0.00131002050863763\\
0.1	0.00131002050863762\\
0.11	0.00131002050863763\\
0.12	0.0156878569358224\\
0.13	0.0356878569358224\\
0.14	0.0381520424476946\\
0.15	0.0381520424476946\\
0.16	0.0381520424476946\\
0.17	0.0381520424476946\\
0.18	0.0381520424476946\\
0.19	0.0381520424476946\\
0.2	0.0381520424476946\\
0.21	0.0381520424476946\\
0.22	0.0381520424476946\\
0.23	0.0381520424476946\\
0.24	0.0381520424476946\\
0.25	0.0381520424476946\\
0.26	0.0381520424476946\\
0.27	0.0381520424476946\\
0.28	0.0381520424476946\\
0.29	0.0381520424476946\\
0.3	0.0381520424476946\\
0.31	0.0578335551765226\\
0.32	0.0778335551765226\\
0.33	0.0978335551765226\\
0.34	0.117833555176523\\
0.35	0.137833555176523\\
0.36	0.157833555176523\\
0.37	0.177833555176523\\
0.38	0.197833555176523\\
0.39	0.205891132094649\\
0.4	0.205891132094649\\
0.41	0.205891132094649\\
0.42	0.205891132094649\\
0.43	0.205891132094649\\
0.44	0.205891132094649\\
0.45	0.205891132094649\\
0.46	0.205891132094649\\
0.47	0.205891132094649\\
0.48	0.205891132094649\\
0.49	0.205891132094649\\
0.5	0.205891132094649\\
0.51	0.205891132094649\\
0.52	0.205891132094649\\
0.53	0.205891132094649\\
0.54	0.214797116047324\\
0.55	0.234797116047324\\
0.56	0.254797116047324\\
0.57	0.274797116047324\\
0.58	0.294797116047325\\
0.59	0.314797116047324\\
0.6	0.334797116047324\\
0.61	0.354797116047324\\
0.62	0.374797116047324\\
0.63	0.394797116047324\\
0.64	0.414797116047324\\
0.65	0.434797116047324\\
0.66	0.454797116047324\\
0.67	0.474797116047324\\
0.68	0.4782969\\
0.69	0.4782969\\
0.7	0.4782969\\
0.71	0.4782969\\
0.72	0.4782969\\
0.73	0.4782969\\
0.74	0.4782969\\
0.75	0.4782969\\
0.76	0.49464845\\
0.77	0.51464845\\
0.78	0.53464845\\
0.79	0.55464845\\
0.8	0.57464845\\
0.81	0.59464845\\
0.82	0.61464845\\
0.83	0.63464845\\
0.84	0.65464845\\
0.85	0.67464845\\
0.86	0.69464845\\
0.87	0.71464845\\
0.88	0.729\\
0.89	0.729\\
0.9	0.729\\
0.91	0.7345\\
0.92	0.7545\\
0.93	0.7745\\
0.94	0.7945\\
0.95	0.8145\\
0.96	0.8345\\
0.97	0.8545\\
0.98	0.8745\\
0.99	0.8945\\
1	0.9\\
};
\addlegendentry{$f_t$}

\addplot[only marks, mark=*, mark options={}, mark size=1.5000pt, draw=mycolor1, fill=mycolor1] table[row sep=crcr]{%
x	y\\
1	0.9\\
0.9	0.729\\
0.729	0.4782969\\
0.4782969	0.205891132094649\\
0.205891132094649	0.0381520424476946\\
0.0381520424476946	0.00131002050863762\\
0.00131002050863762	1.54453835974606e-06\\
1.54453835974606e-06	2.14703887025434e-12\\
2.14703887025434e-12	4.14879831934474e-24\\
};
%\addlegendentry{Data points}

\end{axis}
\end{tikzpicture}%
}
\end{figure}
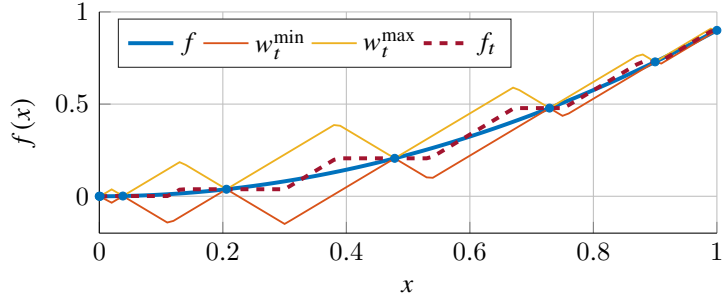
\begin{proof}
    In the following, we focus on the more elaborated case where $w_{t}^{\max}(z)\neq w_{t+1}^{\max}(z)$ and $w_{t}^{\min}(z)\neq w_{t+1}^{\min}(z)$. The alternative case can be trivially shown by following similar steps.
\begin{align*}
& \, \| w_{t}^{\max}(z)-w_{t}^{\min}(z) - (w_{t+1}^{\max}(z)-w_{t+1}^{\min}(z))\| \\ \stackrel{\eqref{eq:PF_kinky_bounds_max} }{=} & \,
     \| \min_{y\in\D_{t}} [\fw(y) + \1 q \|z-y\|^\lambda] - \max_{y\in\D_{t}} [\fw(y) - \1 q \|z-y\|^\lambda]\\
     & \, -\min \{w_{t}^{\max}(z), f(z_{t}) + \1 q \|z-z_{t}\|^\lambda \} + \max \{w_{t}^{\min}(z), f(z_{t}) - \1 q \|z-z_{t}\|^\lambda \} \| \\
     \stackrel{(*)}{\leq} & \, \|  \min_{y\in\D_{t}} [\fw(y) + \1 q \|z_t-y\|^\lambda+ \1 q \|z-z_t\|^\lambda] - \max_{y\in\D_{t}} [\fw(y) - \1 q \|z_t-y\|^\lambda -  \1 q \|z-z_t\|^\lambda] \\
     & \, -  f(z_{t}) - \1 q \|z-z_{t}\|^\lambda  + f(z_{t}) - \1 q \|z-z_{t}\|^\lambda \| \\
     =  & \,  \| \min_{y\in\D_{t}} [\fw(y) + \1 q \|z_t-y\|^\lambda]- \max_{y\in\D_{t}} [\fw(y) - \1 q \|z_t-y\|^\lambda ] + 2 \cdot\1 q \|z-z_t\|^\lambda - 2\cdot \1 q \|z-z_{t}\|^\lambda \| \\
     \stackrel{\eqref{eq:PF_kinky_bounds_max} }{=} & \, \| w_{t}^{\max}(z_t)-w_{t}^{\min}(z_t)\|.
\end{align*}
where in ($*$) we use the subadditivity of the function $\|\cdot\|^\lambda$ and that $w_{t}^{\max}(z)\neq w_{t+1}^{\max}(z)$ and $w_{t}^{\min}(z)\neq w_{t+1}^{\min}(z)$.
Since the bound above holds for any $z\in \ZZ$, then we have that
\begin{align*}
    \max_{y\in \ZZ} \| w_{t}^{\max}(y)-w_{t}^{\min}(y) - (w_{t+1}^{\max}(y)-w_{t+1}^{\min}(y)) \|) = \| w_{t}^{\max}(z_t)-w_{t}^{\min}(z_t)\|,
\end{align*}
which concludes the proof.
\end{proof}
\color{black}

Based on the confidence bounds, the uncertainty function $\w_t$ is chosen as follows
\begin{align}
\w_t(z) &:= \left\{ w\in\R^n : w_t^\text{min}(z) \leq w + \f_t(z) \leq w_t^\text{max}(z) \right\}.
\label{eq:PF_kinky_uncertainty_2}
\end{align}
In \cite{calliess2014conservative}, it is shown that
\begin{subequations}\label{eq:PF_max_uncertainty_bounded}
\begin{align}
\f_{t+1}(x,u)\oplus \w_{t+1}(x,u)\subseteq \f_t(x,u)\oplus \w_t(x,u), ~\forall (x, u)\in \ZZ, \label{eq:PF_max_uncertainty_bounded_2}\\
\w_t(z)=\left\{ 0 \right\}, ~ \forall z\in\D_t, ~  t\in\N, \label{eq:PF_max_uncertainty_bounded_0} \\
\fw(x,u)-\f_t(x,u)\in\w_t(x,u), ~ \forall (x, u)\in \ZZ, t\in \N. \label{eq:PF_max_uncertainty_bounded_1}
\end{align}
\end{subequations}
Note in particular that \cref{eq:PF_max_uncertainty_bounded_0} implies that the uncertainty is zero at all data-points that have been observed in the past.
%

%
%%%%%%%%%%%%%%%%%%%%%%%%%%%%%%%%%%%%%%%%%%%%%%%%%%%%%%%%%%%%%%%%%%%%%%%%%%%%%%%%%%%%%%%%%%%%%%%%%%%%%%%%%%%%%%
\paragraph*{Model Predictive Control scheme: }%
\label{subsection:PF_MPC}
Our goal is to steer system \cref{eq:PF_true_dynamics} from some initial condition $x_0\in\mathscr{X}$ to the origin by choosing an appropriate input sequence $u_t$, for each $t\in\N$. To achieve this, we formulate a model predictive control scheme that employs the nominal dynamics $\f_t$ to predict the future evolution of the system, and leverages the uncertainty bounds $\w_t$ to guarantee robust constraint satisfaction at each time-step. 

Consider the following finite horizon \emph{cost function} $V:\R^{n\times N} \times \R^{m\times (N-1)}\to\Rp$ defined as
\begin{align}
\label{eq:PF_cost}
V(x_{\cdot|t},u_{\cdot|t}) = V_f(x_{N|t}) + \sum_{k=0}^{N-1} \l(x_{k|t},u_{k|t}),
\end{align}
where $\l:\R^n \times \R^m \to \Rp$ is  the \emph{stage cost}, while $V_f:\R^n \to \Rp$ is the \emph{terminal cost}, and $N\in\Np$ is the prediction horizon. The stage cost $\ell$ is chosen to satisfy the following.
\begin{assumption}%
\label{ass:PF_cost_assumptions}
The stage cost $\l$ is continuous, and satisfies
\begin{align*}
\alphaeleven(\|x\|) \leq \inf_{u\in\U} \l(x,u) \leq \alphatwelve(\|x\|),
\end{align*}
for some $\alphaeleven,\alphatwelve\in\kinf$ and for all $x\in\R^n$.
\end{assumption}

\cref{ass:PF_cost_assumptions} is satisfied if $\l(x,u):=x^{\top}Qx+u^{\top}Ru$, where $Q,R \succ 0$ are matrices of appropriate dimension, \cite{soloperto2022nonlinear}. 

\begin{assumption}%
\label{ass:PF_terminal_cost}
There exist a terminal cost $V_f:\mathbb{R}^n\rightarrow \Rp$, terminal controller $\kappa_f:\R^n\to\R^m$, terminal region $\XX_f\subseteq \mathcal{X}$, and functions $\alphathree, \alphafour\in \mathcal{K}_{\infty}$, such that, for all $x\in\XX_f$ and for all $w \in \w_0(x, \kappa_f(x))$, it holds that
\begin{subequations}
\label{eq:terminal_conditions}
\begin{align}
\label{eq:terminal_region}
\f_0(x,\kappa_f(x))+w \subseteq \XX_f, ~~ (x,\kappa_f(x))\in\ZZ, \\
\alphathree(\|x\|) \leq V_f(x) \leq \alphafour(\|x\|), \\
V_f(\f_0(x,\kappa_f(x))+w)-V_f(x) \leq -\l (x,\kappa_f(x)),\label{eq:terminal_cost_decrease}
\end{align}
\end{subequations}
\end{assumption}
Assumption \ref{ass:PF_terminal_cost} is satisfied if the origin is exponentially stable with a common Lyapunov function, \cite{chen1998quasi}.

\begin{assumption}
\label{ass:tubes}
    Given a nominal model $f_t$, an uncertainty function $\w_t$, and a state and input pair $(x,u)\in \ZZ$, it is possible to construct tubes $\XX(x,u,t)\subseteq \X$ such that
    \begin{align*}
    \f_t(x,u) \oplus \w_t(x,u) \in\XX(x,u,t), \quad \forall t\in \N, \\
    f_{t+1}(x,u)\oplus \w_{t+1} \subseteq f_t(x,u)\oplus \w_{t}\Rightarrow \XX(x,u,t+1) \subseteq \XX(x,u,t).
    \end{align*}
    \end{assumption}
Assumption \ref{ass:tubes} is common in robust MPC approaches that only consider the initial knowledge of the system, \cite{kohler2020computationally}. In our case, satisfying it becomes ``easier'' over time, due to the monotonicity in uncertainty implied by \cref{eq:PF_max_uncertainty_bounded_2}.

Given a nominal model $f_t$, a description of the uncertainty $\w_t$, a cost function $V$, and an appropriately designed terminal set $\mathcal{X}_f$, we consider the following finite-horizon optimal control problem
\begin{subequations}
\label{eq:PF_MPC}
\renewcommand{\arraystretch}{1.25}
\begin{align}%
& \underset{\XX_{\cdot|t},x_{\cdot|t},u_{\cdot|t}}{\text{minimize}} \quad V(x_{\cdot|t},u_{\cdot|t}) & \notag \\
& \begin{array}[c]{@{}l@{\hspace{\len}}l}
\text{subject to} %
& x_{k+1|t} = \f_t(x_{k|t},u_{k|t}), \\
& x_{0|t} = x_t,
\end{array}&
%-----------------------------------------------------------
% Begin parenthesis ------------------------------------------
\left.\vphantom{%
\begin{array}[c]{@{}l@{\hspace{\len}}l}
\text{subject to} %
& x_{k+1|t} = \f_t(x_{k|t},u_{k|t}), \\
& x_{0|t} = x_t,
\end{array}}%
\right\}%
\label{eq:PF_MPC_1}\\
% End parenthesis ------------------------------------------
%-----------------------------------------------------------
& \begin{array}[c]{@{}l@{\hspace{\len}}l}
\phantom{\text{subject to}} %
& \f_t(\bar{x}_{k|t},u_{k|t})+w_{k|t}\in\XX_{k+1|t}, \\
& (\bar{x}_{k|t},u_{k|t}) \in \ZZ, \\
& \forall w_{k|t}\in\w_t(\bar{x}_{k|t},u_{k|t}) , \\
& \forall \bar{x}_{k|t}\in\XX_{k|t}, \\
& k=0,\dots,N-1,
\end{array}&
%-----------------------------------------------------------
% Begin parenthesis ------------------------------------------
\left.\vphantom{%
\begin{array}[c]{@{}l@{\hspace{\len}}l}
\phantom{\text{subject to}} %
& \f_t(x_{k|t},u_{k|t})+w_{k|t}\in\XX_{k+1|t}, \\
& (x_{k|t},u_{k|t}) \in \ZZ, \\
& \forall w_{k|t}\in\w_t(x_{k|t},u_{k|t}) , \\
& \forall x_{k|t}\in\XX_{k|t}, \\
& k=0,\dots,N-1,
\end{array}}%
\right\}%
\label{eq:PF_MPC_2}\\
% End parenthesis ------------------------------------------
%-----------------------------------------------------------
& \phantom{\text{subject to}} \quad x_t \in \XX_{0|t}, \quad
 \XX_{N|t} \subseteq \XX_f, \label{eq:PF_MPC_3}
%-----------------------------------------------------------
\end{align}
\end{subequations}
%
%
%\begin{figure}[b]
%\centering
%\input{Figures/gp.tex}
%\caption{Kinky inference bounds (thin continuous lines) and mean function (dashed line, in red) compared againts the true function (thick and continuous, in blue) for a one-dimensional system.}
%\label{fig:PF_gp_figure}
%\end{figure}
%
%\begin{assumption}
%\label{ass:tubes}
%    Given a nominal model $f_t$, an uncertainty function $\w_t$, and a state and input pair $(x,u)\in \ZZ$, it is possible to construct a tube $\XX_{k}\subseteq \XX$ such that
    %\begin{align*}
    %\f_t(x,u) \oplus \w_t(x,u) \in\XX(x, u, t), \quad \forall t\in \N, \\
    %f_{t+1}(x,u)\oplus \w_{t+1} \subseteq f_t(x,u)\oplus \w_{t}\Rightarrow \XX(x, u, t+1) \subseteq \XX(x, u, t).
    %\end{align*}
    %\end{assumption}
%
where $\XX_{\cdot|t}$ are appropriately constructed tubes that satisfy Assumption \ref{ass:tubes}. The optimization problem \cref{eq:PF_MPC} is a general description of a robust MPC scheme, and is not implementable without a proper construction of the sets $\XX_{\cdot|t}$. In \cite{soloperto2019collision}, it is shown how \cref{eq:PF_MPC} can describe several approaches, including MPC schemes for linear systems subject to bounded additive disturbance, \cite{chisci2001systems} and \cite{mayne2005robust}, and nonlinear system subject to parametric uncertainties, \cite{kohler2021robust}.

We denote the optimizers of the Problem \cref{eq:PF_MPC} by $(\XX^*_{\cdot|t},x_{\cdot|t}^*,u_{\cdot|t}^*) = (\XX^*_{\cdot|t},z_{\cdot|t}^*)$, and the optimal value by
\begin{align*}
V_t^*(x_t) := V_f(x_{N|t}^*) + \sum_{k=0}^{N-1} \l(x_{k|t}^*,u_{k|t}^*).
\end{align*}

The equality constraints \cref{eq:PF_MPC_1} produce a nominal trajectory $x_{\cdot|t}$, obtained by propagating the initial condition $x_t$ through the nominal model $\f_t$ with the input $u_{\cdot|t}$. The constraints in \cref{eq:PF_MPC_2} are designed to ensure that the closed-loop trajectory of the system satisfies the constraints $(x_t,u_t)\in\ZZ$. This is achieved thanks to the introduction of the tubes $\XX_{\cdot|t}\subset\R^n$, which are guaranteed to contain the true state of the system at any time step if the inputs $u_{\cdot|t}$ were to be applied.

The closed-loop system can be obtained by combining the optimal control input $u_{0|t}^*$ with the dynamics \cref{eq:PF_true_dynamics} using the receding-horizon paradigm:
\begin{align}%
\label{eq:PF_closed_loop}
x_{t+1}=f(x_t,u_t),~~ u_t=u_{0|t}^*.
\end{align}

% In \cref{subsection:APP_terminal_controller} we prove that \cref{eq:terminal_cost_decrease} is satisfied if the linear time-invariant system obtained by linearizing \cref{eq:PF_true_dynamics} around the origin is stabilizable, $\fw$ is locally Lipschitz around the origin, if $h_t(z)$ is sufficiently small for all $z$ in a neighborhood of the origin, and $f_t$ is continuously differentiable around the origin for all $t\in\N$.
\section{Theoretical analysis}%
\label{section:TA}
In this section, we show that the closed-loop system converges to the origin despite the presence of model uncertainty.
Let
\begin{align}
    h_t(z) &:= \| w^\text{max}_t(z)-w^\text{min}_t(z) \|. 
\end{align}
The function $h_t$ can be interpreted as the level of uncertainty about the system $\fw(z)$ at a given time $t$, where a small value of $h_t(z)$ implies accurate knowledge of $\fw(z)$. Based on \eqref{eq:PF_kinky_bounds_max} and \eqref{eq:PF_max_uncertainty_bounded}, we have that $h_t(z)=0$ for all the previously visited points $z\in\mathcal{D}_t$.

%The proof of \hyperref[lemma:PF_kinky_satisfies_model_assumptions]{Lemma \ref*{lemma:PF_kinky_satisfies_model_assumptions}} is shown in \cref{subsection:APP_kinky_satisfies_model_assumptions}.

We define the \emph{uncertainty size} $C_t\in\Rp$ across the entire space $\ZZ$ as
\begin{align}%
\label{eq:PF_uncertainty_size}
C_t:=\int_{z\in\ZZ} h_t(z)dz.
\end{align}%
According to \eqref{eq:PF_max_uncertainty_bounded}, it is easy to verify that $0 \leq C_{t+1}\leq C_t < \infty$ for all $t\in\N$. The upper-bound is ensured since $\ZZ$ is compact and $h_0(z) < \infty$ for all $z\in\ZZ$.

\begin{lemma}%
\label{lemma:PF_uncertainty_size_2}
Under \cref{ass:PF_system_assumption}, there exists a function $\alphafive\in\kinf$ such that $C_{t+1}-C_t \leq -\alphafive(h_t(x_t,u_t))$ for all $t\in\N$.
\end{lemma}
\begin{proof}
We begin by applying Lemma 3 of \cite{soloperto2022guaranteed} to the function $h_t-h_{t+1}$. To this end, consider that $h_t-h_{t+1}$ is continuous in $z$ (since both $h_t$ and $h_{t+1}$ are) and, in particular, it is uniformly continuous in $z$ when restricted to $\ZZ$. As a result, there exists some $\deltaone\in\kinf$ such that
\begin{equation*}
|h_t(z)-h_{t+1}(z) - [h_t(y)-h_{t+1}(y)] | \leq \deltaone(\|z-y\|) ,~~~ \forall z,y\in\ZZ,
\end{equation*}
and we can therefore apply Lemma 3 of \cite{soloperto2022guaranteed}, which yields
\begin{align*}
h_t(\bar{z})-h_{t+1}(\bar{z}) \leq \deltatwo \left( \int_{z\in\ZZ} [h_t(z)-h_{t+1}(z)]dz \right),
\end{align*}
for some $\deltatwo\in\kinf$ and for all $\bar{z}\in\ZZ$. The inequality holds if we apply the function $\deltatwo^{-1}\in\kinf$ on both sides of the inequality (recall that the inverse of a $\kinf$ function exists and is itself $\kinf$, as stated in page 4, \cite{kellett2014compendium}), obtaining for all $\bar{z}\in\ZZ$
\begin{align*}
\deltatwo^{-1} (h_t(\bar{z})-h_{t+1}(\bar{z})) \leq \int_{z\in\ZZ} [h_t(z)-h_{t+1}(z)]dz.
\end{align*}
Choosing $\bar{z}=z_t$, we have $h_{t+1}(z_t)=0$ from \cref{eq:PF_max_uncertainty_bounded_0}, and therefore
\begin{align*}
C_{t+1}-C_t = \int_{z\in\ZZ} h_{t+1}(z) dz - \int_{z\in\ZZ} h_{t}(z) dz \leq -\deltatwo^{-1}(h_t(z_t)) =:  -\alphafive(h_t(z_t)).
\end{align*}
\end{proof}
%The proof of \hyperref[lemma:PF_uncertainty_size_2]{Lemma \ref*{lemma:PF_uncertainty_size_2}} is shown in \cref{section:APP_proof_lemma_4}.
%
\begin{theorem}
\label{theorem:TA_stability}
Let Assumptions \ref{ass:PF_system_assumption}, \ref{ass:PF_cost_assumptions}, \ref{ass:PF_terminal_cost}, and \ref{ass:tubes} hold, and assume that the MPC problem \cref{eq:PF_MPC} is feasible at time $t=0$. Then, the MPC scheme \cref{eq:PF_MPC} is feasible for all time $t\in\N_+$, and the closed-loop system \cref{eq:PF_closed_loop} asymptotically converges to the origin while satisfying the state and input constraints \eqref{eq:constraints}.
\end{theorem}
\begin{proof}
The proof of Theorem \ref{theorem:TA_stability} is divided into three parts: in part a) we start by showing recursive feasibility of \eqref{eq:PF_MPC}, in part b) we then show that the origin is input-to-state stable, i.e., it holds that
\begin{align}
\label{eq:ISS}
V_{t+1}^*(x_{t+1})- V_t^*(x_t) \leq - \ell (x_t,u_t)+\alphasix (h_t(z_t)), \quad \forall t\in \Np,
\end{align}
for some $\alphasix \in \mathcal{K}_\infty$, and finally in c) we show that $h_t(z_t)\rightarrow 0$ for $t\rightarrow \infty$. The combination of the last two points implies that the system asymptotically converges to the origin.
\paragraph{a) Recursive feasibility:}
We introduce two different state and input trajectories. One, denoted by $(\hat{x}_{\cdot|t+1}, \hat{u}_{\cdot|t+1})$, considers the case where the nominal model is not updated, i.e., $f_{t+1}=f_t$, while the other, denoted by $(\bar{x}_{\cdot|t+1}, \hat{u}_{\cdot|t+1})$, the case where the nominal model is updated. Note that the latter is the only trajectory that is guaranteed to be feasible in Problem \eqref{eq:PF_MPC} at time $t+1$ and that both trajectories share the same input $\hat{u}_{\cdot|t+1}$.

Let define $(\hat{x}_{\cdot|t+1}, \hat{u}_{\cdot|t+1})$ as
\begin{align}
\label{eq:candidate_nominal}
    \hat{u}_{k|t+1}= u^*_{k+1|t}, \quad  k=0,\dots ,N-2, \quad
    \hat{u}_{N-1|t+1}= \kappa _f(x^*_{N|t}),\nonumber\\
\hat{x}_{0|t+1} = f_t(x_t, u_t), \quad \hat{x}_{k|t+1} = f_t(\hat{x}_{k-1|t+1}, \hat{u}_{k-1|t+1}), \quad k=1,\dots ,N-1.
\end{align}
The trajectory $(\hat{x}_{\cdot |t+1},\hat{u}_{\cdot |t+1})$ satisfies the state and input constraints in \eqref{eq:PF_MPC} at time $t+1$, provided that the initial condition is $x_{t+1}=f_t(x_t,u_t)$. Likewise, we define $(\bar{x}_{\cdot|t+1},\hat{u}_{\cdot|t+1})$ as
\begin{align}
\label{eq:candidate_updated}
    \bar{x}_{k+1|t+1}= f_{t+1}(\bar{x}_{k|t+1},\hat{u}_{k|t+1}),  \quad k=0,\dots ,N-1, \quad
    \bar{x}_{0|t+1}=f_{t+1}(x_t,u_t) = x_{t+1}. 
\end{align}
Note that $(\bar{x}_{\cdot|t+1},\hat{u}_{\cdot|t+1})$ is obtained by propagating the actual initial condition $x_{t+1}=f(x_t, u_t)$ through the updated dynamics $f_{t+1}$ under the input $\hat{u}_{\cdot|t+1}$, defined in \eqref{eq:candidate_nominal}.

Thanks to \cref{eq:PF_max_uncertainty_bounded} and \eqref{eq:terminal_conditions}, we have that the trajectory $(\bar{x}_{\cdot |t+1},\hat{u}_{\cdot |t+1})$ is feasible at time $t+1$ for Problem \eqref{eq:PF_MPC}. The existence of candidate tubes is ensured based on Assumption \ref{ass:tubes}.
\paragraph{b) Input-to-state stability:}
Starting from $k=0$, we have 
\begin{align}
\label{eq:upper_bound_k0}
\| \bar{x}_{0|t+1}-\hat{x}_{0|t+1} \| \stackrel{\cref{eq:candidate_nominal,eq:candidate_updated}}{ =} \| f_{t+1}(x_t,u_{t})-f_t(x_t,u_{t}) \| \stackrel{\eqref{eq:TA_model_updates_bounded_1}}{\leq} \gamma_0(h_t(z_t)),
\end{align}
where $\gamma_0 := \alphatwo$. Then, for $k=1, \dots, N$, we can use induction to show that
\begin{align}
\label{eq:TA_ISS_proof_2}
& \| \bar{x}_{k|t+1}-\hat{x}_{k|t+1} \| 
\stackrel{\cref{eq:candidate_nominal,eq:candidate_updated}}{ =} \| f_{t+1}(\bar{x}_{k-1|t+1},\hat{u}_{k-1|t+1})-f_t(\hat{x}_{k-1|t+1},\hat{u}_{k-1|t+1}) \|  \\
\stackrel{\hphantom{(5b),\,(13)}}{\leq} & \| f_{t+1}(\bar{x}_{k-1|t+1},\hat{u}_{k-1|t+1})-f_{t+1}(\hat{x}_{k-1|t+1},\hat{u}_{k-1|t+1})\| \nonumber \\ & + \|f_{t+1}(\hat{x}_{k-1|t+1},\hat{u}_{k-1|t+1})-f_t(\hat{x}_{k-1|t+1},\hat{u}_{k-1|t+1}) \| \nonumber \\
\stackrel{\eqref{eq:TA_model_updates_bounded_1},\,\eqref{eq:PF_nominal_model_1}}{\leq} & \, \alphaone( \| \bar{x}_{k-1|t+1}-\hat{x}_{k-1|t+1} \|) + \alphatwo(h_t(z_t)) 
\stackrel{\eqref{eq:upper_bound_k0}}{=} (\alphaone\circ\gamma_k + \alphatwo) (h_t(z_t)) 
\stackrel{\hphantom{(a)}}{=:}  \gamma _{k+1}(h_t(z_t)), \nonumber
\end{align}
where for $k \geq 0$ we recursively define $\gamma _{k+1}:= \alphaone\circ\gamma_k + \alphatwo$. Since class $\kinf$ functions are closed under composition and summation, \cite{kellett2014compendium}, we have $\gamma_{k+1} \in \kinf$. Therefore, from \cref{eq:TA_ISS_proof_2}, we have that there always exists some $\gamma_k\in \mathcal{K}_\infty$ such that for $k=0,1,\dots,N$ and for all $t\in\N$
\begin{align}
\label{eq:TA_ISS_proof_4}
\| \bar{x}_{k|t+1}-\hat{x}_{t|k+1} \| \leq \gamma _k(h_t(z_t)).
\end{align}
Next, since the stage cost $\ell$ is continuous and $\X$ is compact, there exists some $\alphaseven\in\kinf$ such that
\begin{align}
\label{eq:TA_ISS_proof_6}
\|\ell(x,u)-\ell(y,u)\| \leq \alphaseven(\|x-y\|),
\end{align}
for all $x,y\in\X$, \cite{limon2009input}, Lemma 1. We now upper-bound the following difference
\begin{align}
\label{eq:TA_ISS_proof_5}
% \sum_{k=0}^{N-2} \ell (\bar{x}_{k|t+1},\hat{u}_{k|t+1})-\ell (\hat{x}_{k|t+1},\hat{u}_{k|t+1}) & \stackrel{\cref{eq:TA_ISS_proof_6}}{\leq} \sum_{k=0}^{N-2} \alphaseven(\| \bar{x}_{k|t+1}-\hat{x}_{k|t+1} \|) \nonumber \\
% & \stackrel{\cref{eq:TA_ISS_proof_4}}{\leq} \sum_{k=0}^{N-2} (\alphaseven\circ\gamma_k)(h_t(z_t)) =: \alphaeight(h_t(z_t)),
\sum_{k=0}^{N-2} \ell (\bar{x}_{k|t+1},\hat{u}_{k|t+1})-\ell (\hat{x}_{k|t+1},\hat{u}_{k|t+1}) & \stackrel{\cref{eq:TA_ISS_proof_6,eq:TA_ISS_proof_4}}{\leq} \sum_{k=0}^{N-2} (\alphaseven\circ\gamma_k)(h_t(z_t)) =: \alphaeight(h_t(z_t)),
\end{align}
where $\alphaeight:=\sum_{k=0}^{N-2} \alphaseven\circ\gamma_k\in\kinf$. Since by optimality we have
 that $V_{t+1}^*(x_{t+1}) \leq V(\bar{x}_{\cdot |t+1},\hat{u}_{\cdot |t+1})$,
we can make use of \cref{eq:TA_ISS_proof_5} to obtain
\begin{align}
& \, V_{t+1}^*(x_{t+1}) - V_t^*(x_t) \leq V(\bar{x}_{\cdot |t+1}, \hat{u}_{\cdot |t+1})-V_t^*(x_t) \nonumber\\
\stackrel{\eqref{eq:TA_ISS_proof_5}}{\leq} & - \ell (x_t,u_t) + \alphaeight(h_t(z_t)) + \ell (\bar{x}_{N-1|t+1},\hat{u}_{N-1|t+1}) - V_f(x^*_{N|t}) +V_f(\bar{x}_{N|t+1}) \label{eq:TA_ISS_proof_7}
\end{align}
Since $V_f$ is continuous and $\X$ is compact, there exists a function $\alphanine\in\kinf$ such that
\begin{align}
\label{eq:TA_ISS_proof_9}
\|V_f(x)-V_f(y)\| \leq \alphanine(\|x-y\|),
\end{align}
for all $x,y\in\X$. In the following, we upper-bound the sum of the last three terms in \cref{eq:TA_ISS_proof_7}:
\begin{align}
& \ell(\bar{x}_{N-1|t+1},\hat{u}_{N-1|t+1})-V_f(x^*_{N|t}) +V_f(\bar{x}_{N|t+1}) \nonumber \\
\stackrel{\hphantom{(c)}}{=}\,&\ell(\bar{x}_{N-1|t+1},\hat{u}_{N-1|t+1}) - \ell(\hat{x}_{N-1|t+1},\hat{u}_{N-1|t+1}) + \ell(\hat{x}_{N-1|t+1},\hat{u}_{N-1|t+1}) - V_f(x^*_{N|t}) \nonumber \\
& + V_f(\hat{x}_{N|t+1}) + V_f(\bar{x}_{N|t+1}) - V_f(\hat{x}_{N|t+1}) \nonumber \\
& \hspace{-2.5em} \stackrel{\cref{eq:TA_ISS_proof_6,eq:TA_ISS_proof_9}}{\leq} \alphaseven(\|\hat{x}_{N-1|t+1}-\bar{x}_{N-1|t+1}\|) + \alphanine (\| \bar{x}_{N|t+1}-\hat{x}_{N|t+1} \|) + \ell(x^*_{N|t},\kappa_f(x^*_{N|t})) \nonumber \\
& - V_f(x^*_{N|t}) + V_f(f_t(x^*_{N|t},\kappa_f(x^*_{N|t}))) \nonumber \\
& \hspace{-2.5em} \stackrel{\eqref{eq:terminal_cost_decrease},\,\eqref{eq:TA_ISS_proof_4}}{\leq} (\alphaseven\circ\gamma_{N-1} + \alphanine\circ\gamma_N) (h_t(z_t)) =: \, \alphaten(h_t(z_t)), \label{eq:TA_ISS_proof_8}
\end{align}
where we defined $\alphaten:= \alphaseven\circ\gamma_{N-1} + \alphanine\circ\gamma_N\in\kinf$. Combining \cref{eq:TA_ISS_proof_7} and \cref{eq:TA_ISS_proof_8}, and choosing $\alphasix:=\alphaeight+\alphaten\in\kinf$, we obtain \cref{eq:ISS}. Since $V_t^*$ is an ISS Lyapunov function, the closed-loop system is input-to-state stable (compare Definition 1 and Theorem 1 in \cite{li2018input}).

\paragraph*{c) Convergence: }
Using Lemma \ref{lemma:PF_uncertainty_size_2}, it holds that
\begin{align*}
\sum_{k=0}^{t} C_{k+1}-C_k \leq - \sum_{k=0}^{t} \alphafive(h_k(z_k)) \Rightarrow \sum_{k=0}^{t} \alphafive(h_k(z_k)) \leq C_0-C_{t+1} \leq C_0.
\end{align*}
Since the inequality above holds for all $t\in \mathbb{N}$, by taking the limit we have that
\begin{equation*}
\lim_{t \to \infty}\sum_{k=0}^{t} \alphafive(h_k(z_k))  \leq C_0 \Rightarrow \ \lim_{t \to \infty} \alphafive(h_t(z_t)) = 0 \iff \lim_{t \to \infty} h_t(z_t)=0.
\end{equation*}
Using the fact that the closed-loop system is ISS, and since the term $h_t$ converges to zero for $t\rightarrow\infty$, we can leverage the converging-input converging-state property of ISS systems, as stated in Page 3, \cite{jiang2001input}, to conclude that $z_t\to 0$ as $t\to\infty$.
\end{proof}
\section{Numerical example}\label{section:EX}
We consider a two-dimensional system described by the following dynamics
\begin{align*}
x_{t+1} =
\underbrace{\begin{bmatrix}
1 & 0.4 \\ 0 & 0.56+0.1x_t^1
\end{bmatrix}x_t +
\begin{bmatrix}
0 \\ 0.4
\end{bmatrix}u_t}_{:=\hat{f}(x_t,u_t)} + 
\begin{bmatrix}
0\\ 0.9 x^1_t \operatorname{exp}(-x^1_t)
\end{bmatrix},
\end{align*}
where $x_t=(x_t^1,x_t^2)\in\R^2$ and $u_t\in\R$. We initialize the system with $x_0=(3,0)$. The input constraints are defined as $\U=\left\{ u\in\R: |u| \leq 2 \right\}$. The goal is to reach the origin while minimizing the stage cost $\ell(x,u) = x^TQx + u^TRu$, where $Q=\text{diag}(1,1)$, and $R=1$.

For simplicity, we do not consider state constraints or a terminal region. It is possible to prove, using \hyperref[lemma:PF_kinky_satisfies_model_assumptions]{Lemma 2}, Proposition 1 in \cite{limon2009input}, and Theorem 4 in \cite{boccia2014stability}, that the system is ISS for a sufficiently large control horizon.

We considered the case where the function $\hat{f}$ is known, while we use kinky inference to learn the unknown function $x^1\mapsto g(x^1):=0.9x^1\exp(-x^1)$. Our nominal model $f_t$ is therefore defined as $f_t(x,u) = \hat{f}(x_t,u_t) + g_t(x^1)$, where $g_0(x^1) = 0$, for all $x^1\in \R$ \rz{and $g_t$ is given by the mean of the upper and lower bounds of the Kinky inference}. We over-approximate the Lipschitz constant of $g$ in the operating range of states by choosing $q=1.5$, $\lambda=1$.

\rz{The MPC problem described here requires solving a nonlinear optimization problem including nonsmooth equality constraints. This is generally challenging. Future work should focus on adapting the theoretical analysis to the smoothed Kinky inference introduced in \cite{manzano2019output}, where the nonsmoothness is not present.}

We compare the performance of our scheme against an MPC scheme where $f_t(x,u)$ is chosen as $f_t(x,u)=\hat{f}(x,u)$. In \cref{fig:EX_state}, it is possible to see that the lack of learning prevents the second scheme from converging to the origin. On the contrary, our adaptive MPC scheme takes advantage of new data to successfully regulate the system to the origin.

\begin{figure}
\centering
% This file was created by matlab2tikz.
%
%The latest updates can be retrieved from
%  http://www.mathworks.com/matlabcentral/fileexchange/22022-matlab2tikz-matlab2tikz
%where you can also make suggestions and rate matlab2tikz.
%
\definecolor{mycolor1}{rgb}{0.00000,0.44700,0.74100}%
\definecolor{mycolor2}{rgb}{0.85000,0.32500,0.09800}%
\begin{tikzpicture}

\begin{axis}[%
width=0.33*7.636in,
height=0.4*2.478in,
at={(0.4*0.688in,0.4*3.725in)},
scale only axis,
xmin=0,
xmax=50,
xlabel style={draw=none},
xlabel={$t$},
ymin=0,
ymax=3,
ylabel style={font=\color{white!15!black}},
ylabel={$x_1$},
axis background/.style={fill=white},
axis x line*=bottom,
axis y line*=left,
xmajorgrids,
ymajorgrids,
label style = {font=\small},
tick label style = {font=\small},
legend style = {font=\small}
]
\addplot [color=mycolor1, dashed, line width=1.0pt, mark size=1.7pt, mark=*, mark options={solid, mycolor1}, forget plot]
  table[row sep=crcr]{%
0	3\\
5	1.57598875351202\\
10	0.875053096121617\\
15	0.523830434935219\\
20	0.280388895502288\\
25	0.112194405965785\\
30	0.035912686497195\\
35	0.00974722220840318\\
40	0.00236094399704327\\
45	0.000552035249753534\\
50	0.000132229290498923\\
};
\addplot [color=mycolor2, dashdotted, line width=1.0pt, mark size=3pt, mark=asterisk, mark options={solid, mycolor2}, forget plot]
  table[row sep=crcr]{%
0	3\\
5	1.73575334964331\\
10	1.33794533563531\\
15	1.30185835095716\\
20	1.30076120465623\\
25	1.26519375498442\\
30	1.22255624017008\\
35	1.19487993631545\\
40	1.18391686041031\\
45	1.18240240332947\\
50	1.18371614485705\\
};
\end{axis}

\begin{axis}[%
width=0.33*7.636in,
height=0.4*2.478in,
at={(5*0.688in,0.4*3.725in)},
scale only axis,
xmin=0,
xmax=50,
xlabel style={font=\color{white!15!black}},
xlabel={$t$},
ymin=-1.5,
ymax=0.5,
ylabel style={font=\color{white!15!black}},
ylabel={$x_2$},
axis background/.style={fill=white},
axis x line*=bottom,
axis y line*=left,
xmajorgrids,
ymajorgrids,
legend style={at={(0.97,0.03)}, anchor=south east, legend cell align=left, align=left, draw=white!15!black},
label style = {font=\small},
tick label style = {font=\small},
legend style = {font=\small}
]
\addplot [color=mycolor1, dashed, line width=1.0pt, mark size=1.7pt, mark=*, mark options={solid, mycolor1}]
  table[row sep=crcr]{%
0	0\\
1	-0.665574915091165\\
2	-1.23796934969391\\
3	-1.18920964279746\\
4	-0.467274208637409\\
5	0.154842062815713\\
10	-0.118763120985799\\
15	-0.228725185533624\\
20	-0.139062301489104\\
25	-0.0529095337221622\\
30	-0.0195486348511248\\
35	-0.00591787548166396\\
40	-0.00148311883335351\\
45	-0.000347898469873632\\
50	-7.73011070668803e-05\\
};
\addlegendentry{With learning}

\addplot [color=mycolor2, dashdotted, line width=1.0pt, mark size=3pt, mark=asterisk, mark options={solid, mycolor2}]
  table[row sep=crcr]{%
0	0\\
1	-0.665574915091165\\
2	-1.23796934563946\\
3	-1.05486088732678\\
4	-0.202211477834318\\
5	0.438021834223722\\
10	0.0239463692399181\\
15	-0.249325467602716\\
20	-0.224049905036908\\
25	-0.09795534504641\\
30	-0.00922118098838803\\
35	0.0200623323742264\\
40	0.0163697782964511\\
45	0.00565515511641657\\
50	-0.000838488562312889\\
};
\addlegendentry{Without learning}

\end{axis}
\end{tikzpicture}%
\caption{State evolution.}
\label{fig:EX_state}
\end{figure}
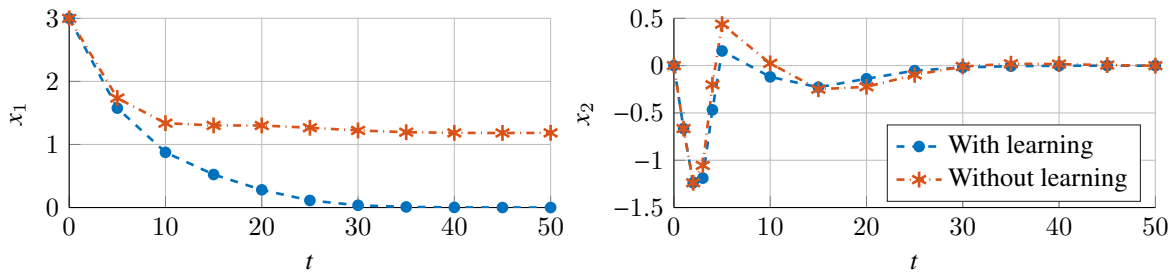
\section{Conclusion}
\label{section:conclusion}
We considered the problem of regulating an unknown nonlinear system using an adaptive model predictive control scheme. Specifically, we considered systems that are continuous and deterministic, while the online-updated nominal models are learned through a kinky inference. We showed that, for the considered class of systems, a standard adaptive MPC scheme, i.e., only with a standard tracking cost function, the system converges, in closed-loop, to the origin. We illustrated our findings in a simulation example.
%\appendix
%\input{Sections/6_appendix.tex}
\acks{The work of Riccardo Zuliani is supported by the Swiss National Science Foundation under NCCR Automation (grant agreement 51NF40\_180545). The work of Raffaele Soloperto is supported by the European Research Council under the H2020 Advanced Grant no. 787845 (OCAL).}
%%%%%%%%%%%%%%%%%%%%%%%%%%%%%%%%%%%%%%%%%%%%%%%%%%%%%%%%%%%%%%%%%%%%%%%%
% Bibliography
%%%%%%%%%%%%%%%%%%%%%%%%%%%%%%%%%%%%%%%%%%%%%%%%%%%%%%%%%%%%%%%%%%%%%%%%
\bibliography{ref.bib}

\begin{thebibliography}{31}
\providecommand{\natexlab}[1]{#1}
\providecommand{\url}[1]{\texttt{#1}}
\expandafter\ifx\csname urlstyle\endcsname\relax
  \providecommand{\doi}[1]{doi: #1}\else
  \providecommand{\doi}{doi: \begingroup \urlstyle{rm}\Url}\fi

\bibitem[Boccia et~al.(2014)Boccia, Gr{\"u}ne, and Worthmann]{boccia2014stability}
Andrea Boccia, Lars Gr{\"u}ne, and Karl Worthmann.
\newblock Stability and feasibility of state constrained mpc without stabilizing terminal constraints.
\newblock \emph{Systems \& control letters}, 72:\penalty0 14--21, 2014.

\bibitem[Bradford et~al.(2020)Bradford, Imsland, Zhang, and del Rio~Chanona]{bradford2020stochastic}
Eric Bradford, Lars Imsland, Dongda Zhang, and Ehecatl~Antonio del Rio~Chanona.
\newblock Stochastic data-driven model predictive control using gaussian processes.
\newblock \emph{Computers \& Chemical Engineering}, 139:\penalty0 106844, 2020.

\bibitem[Calliess(2014)]{calliess2014conservative}
Jan-Peter Calliess.
\newblock \emph{Conservative decision-making and inference in uncertain dynamical systems}.
\newblock PhD thesis, University of Oxford Oxford, 2014.

\bibitem[Calliess et~al.(2020)Calliess, Roberts, Rasmussen, and Maciejowski]{calliess2020lazily}
Jan-Peter Calliess, Stephen~J Roberts, Carl~Edward Rasmussen, and Jan Maciejowski.
\newblock Lazily adapted constant kinky inference for nonparametric regression and model-reference adaptive control.
\newblock \emph{Automatica}, 122:\penalty0 109216, 2020.

\bibitem[Chen and Allg{\"o}wer(1998)]{chen1998quasi}
Hong Chen and Frank Allg{\"o}wer.
\newblock A quasi-infinite horizon nonlinear model predictive control scheme with guaranteed stability.
\newblock \emph{Automatica}, 34\penalty0 (10):\penalty0 1205--1217, 1998.

\bibitem[Chisci et~al.(2001)Chisci, Rossiter, and Zappa]{chisci2001systems}
Luigi Chisci, J~Anthony Rossiter, and Giovanni Zappa.
\newblock Systems with persistent disturbances: predictive control with restricted constraints.
\newblock \emph{Automatica}, 37\penalty0 (7):\penalty0 1019--1028, 2001.

\bibitem[Darby and Nikolaou(2012)]{darby2012mpc}
Mark~L Darby and Michael Nikolaou.
\newblock Mpc: Current practice and challenges.
\newblock \emph{Control Engineering Practice}, 20\penalty0 (4):\penalty0 328--342, 2012.

\bibitem[Fiorenza(2017)]{fiorenza2017holder}
Renato Fiorenza.
\newblock \emph{H{\"o}lder and locally H{\"o}lder Continuous Functions, and Open Sets of Class C\^{} k, C\^{}$\{$k, lambda$\}$}.
\newblock Birkh{\"a}user, 2017.

\bibitem[Forbes et~al.(2015)Forbes, Patwardhan, Hamadah, and Gopaluni]{forbes2015model}
Michael~G Forbes, Rohit~S Patwardhan, Hamza Hamadah, and R~Bhushan Gopaluni.
\newblock Model predictive control in industry: Challenges and opportunities.
\newblock \emph{IFAC-PapersOnLine}, 48\penalty0 (8):\penalty0 531--538, 2015.

\bibitem[Hewing et~al.(2019)Hewing, Kabzan, and Zeilinger]{hewing2019cautious}
Lukas Hewing, Juraj Kabzan, and Melanie~N Zeilinger.
\newblock Cautious model predictive control using gaussian process regression.
\newblock \emph{IEEE Transactions on Control Systems Technology}, 28\penalty0 (6):\penalty0 2736--2743, 2019.

\bibitem[Huang et~al.(2023)Huang, Roberts, and Calliess]{huang2023sample}
Julien~Walden Huang, Stephen~J Roberts, and Jan-Peter Calliess.
\newblock On the sample complexity of lipschitz constant estimation.
\newblock \emph{Transactions on Machine Learning Research}, 2023.

\bibitem[Jiang and Wang(2001)]{jiang2001input}
Zhong-Ping Jiang and Yuan Wang.
\newblock Input-to-state stability for discrete-time nonlinear systems.
\newblock \emph{Automatica}, 37\penalty0 (6):\penalty0 857--869, 2001.

\bibitem[Kellett(2014)]{kellett2014compendium}
Christopher~M Kellett.
\newblock A compendium of comparison function results.
\newblock \emph{Mathematics of Control, Signals, and Systems}, 26:\penalty0 339--374, 2014.

\bibitem[K{\"o}hler et~al.(2020)K{\"o}hler, Soloperto, M{\"u}ller, and Allg{\"o}wer]{kohler2020computationally}
Johannes K{\"o}hler, Raffaele Soloperto, Matthias~A M{\"u}ller, and Frank Allg{\"o}wer.
\newblock A computationally efficient robust model predictive control framework for uncertain nonlinear systems.
\newblock \emph{IEEE Transactions on Automatic Control}, 66\penalty0 (2):\penalty0 794--801, 2020.

\bibitem[K{\"o}hler et~al.(2021)K{\"o}hler, K{\"o}tting, Soloperto, Allg{\"o}wer, and M{\"u}ller]{kohler2021robust}
Johannes K{\"o}hler, Peter K{\"o}tting, Raffaele Soloperto, Frank Allg{\"o}wer, and Matthias~A M{\"u}ller.
\newblock A robust adaptive model predictive control framework for nonlinear uncertain systems.
\newblock \emph{International Journal of Robust and Nonlinear Control}, 31\penalty0 (18):\penalty0 8725--8749, 2021.

\bibitem[Li et~al.(2018)Li, Liu, and Zhang]{li2018input}
Huijuan Li, Anping Liu, and Linli Zhang.
\newblock Input-to-state stability of time-varying nonlinear discrete-time systems via indefinite difference lyapunov functions.
\newblock \emph{ISA transactions}, 77:\penalty0 71--76, 2018.

\bibitem[Limon et~al.(2009)Limon, Alamo, Raimondo, De~La~Pe{\~n}a, Bravo, Ferramosca, and Camacho]{limon2009input}
Daniel Limon, Teodoro Alamo, Davide~M Raimondo, D~Mu{\~n}oz De~La~Pe{\~n}a, Jos{\'e}~Manuel Bravo, Antonio Ferramosca, and Eduardo~F Camacho.
\newblock Input-to-state stability: a unifying framework for robust model predictive control.
\newblock \emph{Nonlinear Model Predictive Control: Towards New Challenging Applications}, pages 1--26, 2009.

\bibitem[Limon et~al.(2017)Limon, Calliess, and Maciejowski]{limon2017learning}
Daniel Limon, J~Calliess, and Jan~Marian Maciejowski.
\newblock Learning-based nonlinear model predictive control.
\newblock \emph{IFAC-PapersOnLine}, 50\penalty0 (1):\penalty0 7769--7776, 2017.

\bibitem[Lorenzen et~al.(2019)Lorenzen, Cannon, and Allg{\"o}wer]{lorenzen2019robust}
Matthias Lorenzen, Mark Cannon, and Frank Allg{\"o}wer.
\newblock Robust mpc with recursive model update.
\newblock \emph{Automatica}, 103:\penalty0 461--471, 2019.

\bibitem[Lu and Cannon(2019)]{lu2019robust}
Xiaonan Lu and Mark Cannon.
\newblock Robust adaptive tube model predictive control.
\newblock In \emph{2019 American Control Conference (ACC)}, pages 3695--3701. IEEE, 2019.

\bibitem[Manzano et~al.(2019)Manzano, Limon, Munoz de~la Pena, and Calliess]{manzano2019output}
Jose~Maria Manzano, Daniel Limon, David Munoz de~la Pena, and Jan~Peter Calliess.
\newblock Output feedback mpc based on smoothed projected kinky inference.
\newblock \emph{IET Control Theory \& Applications}, 13\penalty0 (6):\penalty0 795--805, 2019.

\bibitem[Manzano et~al.(2022)Manzano, Mu{\~n}oz de~la Pe{\~n}a, and Limon]{manzano2022input}
Jose~Maria Manzano, David Mu{\~n}oz de~la Pe{\~n}a, and Daniel Limon.
\newblock Input-to-state stable predictive control based on continuous projected kinky inference.
\newblock \emph{International Journal of Robust and Nonlinear Control}, 2022.

\bibitem[Mayne et~al.(2005)Mayne, Seron, and Rakovi{\'c}]{mayne2005robust}
David~Q Mayne, Mar{\'\i}a~M Seron, and SV~Rakovi{\'c}.
\newblock Robust model predictive control of constrained linear systems with bounded disturbances.
\newblock \emph{Automatica}, 41\penalty0 (2):\penalty0 219--224, 2005.

\bibitem[Mesbah(2018)]{mesbah2018stochastic}
Ali Mesbah.
\newblock Stochastic model predictive control with active uncertainty learning: A survey on dual control.
\newblock \emph{Annual Reviews in Control}, 45:\penalty0 107--117, 2018.

\bibitem[Rasmussen(2003)]{rasmussen2003gaussian}
Carl~Edward Rasmussen.
\newblock Gaussian processes in machine learning.
\newblock In \emph{Summer school on machine learning}, pages 63--71. Springer, 2003.

\bibitem[Rawlings et~al.(2017)Rawlings, Mayne, and Diehl]{rawlings2017model}
James~Blake Rawlings, David~Q Mayne, and Moritz Diehl.
\newblock \emph{Model predictive control: theory, computation, and design}, volume~2.
\newblock Nob Hill Publishing Madison, WI, 2017.

\bibitem[Soloperto et~al.(2019{\natexlab{a}})Soloperto, K{\"o}hler, M{\"u}ller, and Allg{\"o}wer]{soloperto2019dual}
Raffaele Soloperto, Johannes K{\"o}hler, Matthias~A M{\"u}ller, and Frank Allg{\"o}wer.
\newblock Dual adaptive mpc for output tracking of linear systems.
\newblock In \emph{2019 IEEE 58th Conference on Decision and Control (CDC)}, pages 1377--1382. IEEE, 2019{\natexlab{a}}.

\bibitem[Soloperto et~al.(2019{\natexlab{b}})Soloperto, Köhler, Allgöwer, and Müller]{soloperto2019collision}
Raffaele Soloperto, Johannes Köhler, Frank Allgöwer, and Matthias~A. Müller.
\newblock Collision avoidance for uncertain nonlinear systems with moving obstacles using robust model predictive control.
\newblock In \emph{2019 18th European Control Conference (ECC)}, pages 811--817, 2019{\natexlab{b}}.
\newblock \doi{10.23919/ECC.2019.8796049}.

\bibitem[Soloperto et~al.(2022{\natexlab{a}})Soloperto, K{\"o}hler, and Allg{\"o}wer]{soloperto2022nonlinear}
Raffaele Soloperto, Johannes K{\"o}hler, and Frank Allg{\"o}wer.
\newblock A nonlinear mpc scheme for output tracking without terminal ingredients.
\newblock \emph{IEEE Transactions on Automatic Control}, 68\penalty0 (4):\penalty0 2368--2375, 2022{\natexlab{a}}.

\bibitem[Soloperto et~al.(2022{\natexlab{b}})Soloperto, M{\"u}ller, and Allg{\"o}wer]{soloperto2022guaranteed}
Raffaele Soloperto, Matthias~A M{\"u}ller, and Frank Allg{\"o}wer.
\newblock Guaranteed closed-loop learning in model predictive control.
\newblock \emph{IEEE Transactions on Automatic Control}, 68\penalty0 (2):\penalty0 991--1006, 2022{\natexlab{b}}.

\bibitem[Tanaskovic et~al.(2019)Tanaskovic, Fagiano, and Gligorovski]{tanaskovic2019adaptive}
Marko Tanaskovic, Lorenzo Fagiano, and Vojislav Gligorovski.
\newblock Adaptive model predictive control for linear time varying mimo systems.
\newblock \emph{Automatica}, 105:\penalty0 237--245, 2019.

\end{thebibliography}
\end{document}